\documentclass{amsart}[12pt]

\usepackage{amssymb,amsthm}
\usepackage{eucal}
\usepackage[shortlabels]{enumitem}
\usepackage{thmtools}
\usepackage{hyperref}
\usepackage{interval}
\intervalconfig{soft  open  fences}

\declaretheorem[numberwithin=section]{theorem}
\declaretheorem[sibling=theorem]{proposition}
\declaretheorem[sibling=theorem]{lemma}

\declaretheorem[style=definition,sibling=theorem]{definition}

\declaretheorem[style=remark,numbered=no]{remark}
\declaretheorem[sibling=theorem]{fact}

\declaretheorem{conjecture}
\declaretheorem[sibling=conjecture]{question}

\DeclareMathOperator{\cf}{cf}
\DeclareMathOperator{\cof}{cof}
\DeclareMathOperator{\ot}{ot}

\DeclareMathOperator{\dom}{dom}

\newcommand{\seq}[2]{\langle #1 : #2 \rangle}

\newcommand{\Col}{\textup{Col}}

\newcommand{\ON}{\textup{ON}}

\newcommand{\GCH}{\textup{\textsf{GCH}}}

\newcommand{\ZFC}{\textup{\textsf{ZFC}}}

\newcommand{\PCF}{\textup{\textsf{PCF}}}

\newcommand{\s}{\subseteq}

\newcommand{\fr}{\Vdash}

\newcommand{\rest}{\upharpoonright}

\newcommand{\Q}{\mathbb Q}

\renewcommand{\P}{\mathbb P}

\newcommand{\T}{\mathbb T}
\renewcommand{\S}{\mathbb S}

\author{Maxwell Levine}

\title[Compactness of weak square at singulars of uncountable cofinality]{On compactness of weak square at singulars of uncountable cofinality}

\begin{document}

\maketitle

\begin{abstract} Cummings, Foreman, and Magidor proved that Jensen's square principle is non-compact at $\aleph_\omega$, meaning that it is consistent that $\square_{\aleph_n}$ holds for all $n<\omega$ while $\square_{\aleph_\omega}$ fails. We investigate the natural question of whether this phenomenon generalizes to singulars of uncountable cofinality. Surprisingly, we show that under some mild $\PCF$-theoretic hypotheses, the weak square principle $\square_\kappa^*$ is in fact compact at singulars of uncountable cofinality.\end{abstract}


\section{Introduction and Background}

The properties of singulars of uncountable cofinality are notoriously different from those of countable cofinality. A prime example is Silver's theorem that $\GCH$ cannot fail for the first time at a singular of uncountable cofinality. In contrast, Magidor showed that $\GCH$ can fail for the first time at $\aleph_\omega$. There is therefore a natural question of whether this phenomenon generalizes to more complex structures.\footnote{See \cite{LambieHanson-Silver} and \cite{Levine-Mildenberger2021} for recent examples.}

Here we focus on the combinatorial properties of inner models, notably square principles. Jensen originally distilled the principle $\square_\kappa$ (where $\kappa$ is some given cardinal) to study the properties of G{\"o}del's Constructible Universe $L$ \cite{Jensen1972}. Many variations of $\square_\kappa$ have been studied since then. (Precise definitions will be given below, but Cummings-Foreman-Magidor \cite{Cummings-Foreman-Magidor2001} is the canonical reference for this area.) There is in general a tension between square principles and large cardinals, one instance of which is that $\square_\kappa$ fails if $\kappa$ is larger than a supercompact cardinal. Moreover, the failure of $\square_\kappa$ at a singular cardinal $\kappa$ requires considerable consistency strength from large cardinals \cite{Sargsyan2014}. The models of interest in this area realize some compatibility of both square principles and the compactness properties exhibited by large cardinals.

In this paper we will address the compactness of square principles themselves: whether or not $\square_\kappa$ necessarily holds for some cardinal $\kappa$ if $\square_\delta$ holds for sufficiently many cardinals $\delta<\kappa$. Cummings, Foreman, and Magidor proved that it is consistent that $\square_{\aleph_n}$ holds for $1 \le n < \omega$ but that $\square_{\aleph_\omega}$ fails \cite{Cummings-Foreman-Magidor2003}. Later, Krueger improved the result by obtaining a bad scale on $\aleph_\omega$ in a similar model \cite{Krueger2013}. But such results can also go in the other direction: Cummings, Foreman, and Magidor also proved that if $\square_{\aleph_n}$ holds for all $n<\omega$, there is an object that to some extent resembles a $\square_{\aleph_\omega}$-sequence but with a weaker coherence property \cite{Cummings-Foreman-Magidor2004}. The main result of this paper is along these lines:

\begin{theorem}\label{terse-weak-sq} Suppose that $\kappa$ is a singular strong limit of cofinality $\lambda>\omega$ such that for some stationary set $S \subseteq \kappa$, $\square_\delta^*$ holds for all $\delta \in S$ and $\prod_{\delta \in S}\delta^+$ carries a good scale. Then $\square_\kappa^*$ holds.\end{theorem}

This represents some progress on a question raised by Golshani online regarding a supposed Silver's Theorem for special Aronszajn trees \cite{Golshani-Question}: at any cardinal $\delta$, $\square_\delta^*$ is equivalent to the existence of a special $\delta^+$-Aronszajn tree \cite{Cummings2005}.

The difference between the results of Cummings-Foreman-Magidor and \autoref{terse-weak-sq} is that the resulting sequence is fully coherent---not just coherent at points of uncountable cofinality. In other words, we are able to obtain some compactness for a canonical object by obtaining exactly that canonical object in the end. We nonetheless depend on the goodness of scales, as do Cummings-Foreman-Magidor.

Note that the use of stationarity in \autoref{terse-weak-sq} is necessary:

\begin{proposition} Assuming the consistency of a supercompact cardinal, it is consistent that $\square_{\aleph_\alpha}$ holds for all double successor ordinals $\alpha<\omega_1$ and that there is a bad scale on $\aleph_{\omega_1}$, hence $\square_{\aleph_{\omega_1}}$ fails.\end{proposition}

\begin{proof} We work in $V[\Col(\aleph_1,<\kappa)]$ and force with a product of square-adding posets $\prod_{\alpha=\beta+1<\omega_1}\S_{\aleph_{\alpha}}$ to get a model $W$. This model has a bad scale carried by $\aleph_{\omega_1}$ for the following reason: the added squares could be threaded by a product $\prod_{\alpha=\beta+1<\omega_1}\T_{\aleph_{\alpha},\aleph_{\alpha}}$ (where the threads added to the squares originally of length $\aleph_{\alpha+1}$ have length $\aleph_{\alpha}$). This will preserve regularity of $\aleph_{\omega_1+1}^W$ using the fact that if $\tau$ is a regular cardinal such that $\P$ has size $\le \tau$ and $\Q$ is $\tau^+$-distributive, then $\fr_\P ``\Q$ is $\tau^+$-distributive''. Standard lifting arguments then show that there is a bad scale on $\aleph_{\omega_1}^W$ in the extension by the product of threads, but this implies that there is already a bad scale in $W$, and hence that $\square_{\aleph_{\omega_1}}^*$ fails.\end{proof}

We also note that there exists in the literature a contrast to \autoref{terse-weak-sq} in the case of $\aleph_\omega$, even with good scales: 

\begin{fact}\label{contrast-prop} It is consistent that $\square_{\aleph_n}^*$ holds for all $n<\omega$, all scales are good, and $\square_{\aleph_\omega}^*$ fails.\end{fact}

If $\kappa_0 = \aleph_0$ and $\seq{\kappa_n}{1<n<\omega}$ is a sequence of supercompact cardinals in some ground model, then \autoref{contrast-prop} is witnessed in an extension by $\prod_{n<\omega}\Col(\kappa_0,<\kappa_n)$: Magidor and Shelah showed that all scales on $\aleph_\omega$ are good in this model \cite{Magidor1982}, $\square_{\aleph_\omega}^*$ fails because the strong reflection property holds (see Section 4 of \cite{Cummings-Foreman-Magidor2001}), and $\square_{\aleph_n}^*$ holds for all $n<\omega$ because of $\GCH$ by a theorem of Specker.


For the remainder of the introduction, we will focus on definitions. In \autoref{sec-zfc} we will prove \autoref{terse-weak-sq}. At the end of the paper we will prove that a result similar to \autoref{terse-weak-sq} holds for partial squares:


\begin{theorem}\label{almost-silver-partial-sq} Let $\kappa$ be a singular strong limit cardinal of cofinality $\lambda>\omega$. Suppose there is a stationary set $S \subset \kappa$ such that $\square_\delta$ holds for all $\delta \in S$ and such that $\prod_{\delta \in S} \delta^+$ carries a good scale. Then there is a partial square sequence on $\kappa^+ \cap \cof(> \lambda)$.\end{theorem}

\subsection{Definitions}

We define square sequences in terms of a hierarchy introduced by Schimmerling \cite{Schimmerling1995}.


\begin{definition} We say that $\langle\mathcal{C}_\alpha\mid \alpha \in\lim (\kappa^+)\rangle$ is a $\square_{\kappa,\lambda}$-\emph{sequence} if for all limit $\alpha < \kappa^+$:

\begin{enumerate}

\item each $C\in\mathcal{C}_\alpha$ is a club subset of $\alpha$ with $\ot(C)\leq\kappa$;

\item for every $C\in \mathcal{C}_\alpha$, if $\beta\in\lim(C)$, then $C\cap\beta\in\mathcal{C}_\beta$;

\item $1\leq|\mathcal{C}_\alpha|\leq\lambda$.
\end{enumerate}
\end{definition}
The principle $\square_{\kappa, 1}$ is the original $\square_\kappa$, and $\square_{\kappa,\kappa}$ is the weak square, denoted $\square^*_\kappa$.


\begin{definition} If $\mu$ is a cardinal and $S \subset \lim(\mu^+)$ is stationary, then we say that $\seq{C_\alpha}{\alpha \in S}$ is a \emph{partial square sequence} if for all $\alpha \in S$:

\begin{enumerate}

\item $C_\alpha$ is closed and unbounded in $\alpha$;
\item $\ot(C_\alpha) \le \mu$;
\item if $\beta \in S$ and $\gamma \in \lim C_\alpha \cap \lim C_\beta$, then $C_\alpha \cap \gamma = C_\beta \cap \gamma$.

\end{enumerate}\end{definition}

\begin{definition}\

\begin{enumerate}

\item If $\tau$ is a cardinal and $f,g : \tau \to \ON$, then $f <^\ast g$ if there is some $j<\tau$ such that $f(i)<g(i)$ for all $i \ge j$. The analogous definitions hold for $>^*$ and $=^*$.

\item Given a singular cardinal $\kappa$, we say that a strictly increasing sequence $\vec \kappa = \seq{\mu_i}{i<\cf \kappa}$ of regular cardinals converging to $\kappa$ is a \emph{product} when we regard $\prod_{i<\cf \kappa}\mu_i$ as a space.

\item Given a product $\vec \kappa = \prod_{i<\cf \kappa} \mu_i$, a sequence $\seq{f_\alpha}{\alpha<\kappa^+}$ is a \emph{scale} on $\vec \kappa$ if:

\begin{enumerate}

\item for all $\alpha<\kappa^+$, $f_\alpha \in \vec \kappa$, i.e. $f_\alpha(i)<\mu_i$ for all $i<\cf \kappa$;\

\item for all $\beta<\alpha<\kappa^+$,  $f_\alpha <^\ast f_\beta$;

\item for all $g \in \vec \kappa$, there is some $\alpha<\kappa^+$ such that $g<^\ast f_\alpha$ (i.e$.$ $\seq{f_\alpha}{\alpha<\kappa^+}$ is cofinal in the product $\vec \kappa$).

\end{enumerate}

We also say that the product $\vec \kappa$ \emph{carries} $\vec f$.

\item We will use the term \emph{pseudo-scale} for an object resembling a scale that is not necessarily cofinal in the product $\vec \kappa$, i.e$.$ it satisfies (a) and (b) of the previous item. (However, since pseudo-scales are exact upper bounds, they are cofinal in \emph{some} product, see \cite{Handbook-Abraham-Magidor}.)

\item Given a scale (or pseudo-scale) $\vec f = \seq{f_\alpha}{\alpha<\kappa^+}$, $\alpha < \kappa^+$ is \emph{good} if there is some unbounded $A \subset \alpha$ with $\ot A = \cf \alpha$ and some $j<\cf \kappa$ such that for all $i \ge j$, $\seq{f_\beta(i)}{\beta \in A}$ is strictly increasing.


\item If there is a club $D \subset \kappa^+$ such that every $\alpha \in D$ with $\cf \alpha > \cf \kappa$ is a good point of $\vec f$, then $\vec f$ is a \emph{good scale}. An analogous definition applies for good pseudo-scales.


\end{enumerate}\end{definition}

The reason for defining pseudo-scales is that the cofinality clause of the definition of a scale will be largely irrelevant for our purposes.





\section{$\ZFC$ Results}\label{sec-zfc}

In this section we will prove the main results of the paper. We clarify notions of continuity in \autoref{sec-cont}, then we prove \autoref{terse-weak-sq} in \autoref{sec-weaksq}, and then we sketch an analogous theorem for partial squares in \autoref{sec-partsq}.

\subsection{Continuity}\label{sec-cont}

Our goal in this section is to obtain a strong concept of the continuity used by Cummings, Foreman, and Magidor for scales on a singular cardinal $\kappa$ of cofinality $\lambda$. The material concerning points $\alpha$ such that $\cf(\alpha)>\lambda$ is the same as theirs, but we want to consider some issues that arise when $\cf(\alpha) \le \lambda$. Specifically, continuity is trivial if $\cf(\alpha)<\lambda$, and we would like to modify the concept of continuity for the situation where $\cf(\alpha)=\lambda$ so that the square sequences we define are coherent.

Fix a singular $\kappa$ of cofinality $\lambda>\omega$. We will consider some fixed stationary $S \subseteq \lambda$ and a product $\vec \kappa = \prod_{i \in S} \mu_i$. This formulation will be important when we are considering $\alpha \in \kappa^+ \cap \cof(\lambda)$. Fix a pseudo-scale $\vec f$ on $\vec \kappa$.

\begin{proposition}\label{thinnedgoodness} If $\cf \alpha > \cf \kappa$ and $\alpha$ is a good point, then for any cofinal $B \subset \alpha$ with $\ot B = \cf \alpha$, there is some $B^\ast \subseteq B$ such that $B^*$ witnesses goodness of $\alpha$.\end{proposition}

This follows from what is known as ``The Sandwich Argument.''

\begin{proof} Suppose $A \subset \alpha$ witnesses goodness. Let $\tau = \cf \alpha$ and enumerate $A' := \seq{\alpha_\xi}{\xi<\tau} \subset A$ and $B' := \seq{\beta_\xi}{\xi<\tau} \subset B$ in such a way that for all $\xi<\tau$, $f_{\alpha_\xi} \le^* f_{\beta_\xi} <^* f_{\alpha_{\xi+1}}$. Observe that $A'$ also witnesses goodness of $\alpha$ with respect to some $j'$. For each $\xi<\tau$, let $j_\xi \ge j'$ be such that $i \ge j_\xi$ implies $f_{\alpha_\xi}(i) \le f_{\beta_\xi}(i) < f_{\alpha_{\xi+1}}(i)$. Then there is some unbounded $X \subset \tau$ and some $j<\lambda$ such that for all $\xi<\tau$, $j_\xi = j$. Since $j$ also witnesses goodness with respect to $A'$, this means that if $\xi,\eta \in X$ and $\xi < \eta$, then for all $i \ge j$, we have $f_{\beta_\xi}(i) < f_{\alpha_{\xi+1}}(i) \le f_{\alpha_\eta}(i) \le  f_{\beta_\eta}(i)$. We have proved the proposition with $B^* = \seq{\beta_\xi}{\xi \in X}$.\end{proof}

Modulo a short argument, this implies:

\begin{proposition}[See Remark 11.1 in \cite{Cummings2005}]\label{goodness-almost-everywhere} If a product $\vec \kappa$ carries a good scale $\vec f$, then there is a scale $\vec g$ such that \emph{every} $\alpha$ with $\cf \alpha > \cf \kappa$ is a good point of $\vec g$.\end{proposition}


\begin{definition} Suppose $\vec f = \seq{f_\alpha}{\beta<\alpha}$ is a $<^*$-increasing sequence on the product $\vec \kappa = \prod_{i \in S}\mu_i$, and that $A \subset \alpha$ is unbounded for some $\alpha < \kappa^+$ with $\ot A = \cf \alpha$.

\begin{itemize}
\item $\vec f_A$ denotes the function $i \mapsto \sup_{\beta \in A}f_\beta(i)$;
\item if $\cf \alpha = \cf \kappa$ and $A = \seq{\beta_i}{i<\cf \kappa}$, $\vec f^\Delta_A$ denotes the function $i \mapsto \sup_{j<i}f_{\beta_j}(i)$.
\end{itemize}
\end{definition}

\begin{definition} If $f$ and $g$ are functions on a product $\vec \kappa$, we write $f =^*_\Delta g$ if there is a club $C \subseteq \lambda$ such that for all $i \in C \cap S$, $f(i)=g(i)$. The definition for $f<^*_\Delta g$ is analogous.\end{definition}

\begin{definition} A scale $\vec f = \seq{f_\alpha}{\alpha<\kappa^+}$ is \emph{totally continuous} if the following hold:

\begin{itemize}
\item if $\cf \alpha < \cf \kappa$, then for all cofinal $A \subset \alpha$ with $\ot A = \cf \alpha$, $(\vec f \rest \alpha)_A =^\ast f_\alpha$;

\item if $\cf \alpha = \cf \kappa$, then for all \emph{clubs} $A \subset \alpha$ such that $\ot A = \cf \alpha$, we have $f_\alpha =^*_\Delta (\vec f \rest \alpha)_A^\Delta$;

\item if $\cf \alpha > \cf \kappa$, then $\alpha$ is a good point, $f_\alpha$ is an exact upper bound of $\seq{f_\beta}{\beta < \alpha}$, and for all cofinal $A \subset \alpha$ witnessing goodness of $\alpha$, we have $(\vec f \rest \alpha)_A =^* f_\alpha$.
\end{itemize}

Even though these cases are different, we will say \emph{by continuity} if we invoke any of them.
\end{definition}

Now we work towards:

\begin{lemma}\label{totally-cont-scale} If $\cf \kappa = \lambda > \omega$, $S \subset \lambda$ is stationary, and $\vec \kappa = \prod_{i \in S} \mu_i$ is a product of regular cardinals on $\kappa$ that carries a good scale, then it carries a totally continuous good scale.\end{lemma}

Fix a $<^*$-increasing sequence $\vec f = \seq{f_\alpha}{\beta<\alpha}$ on a product $\prod_{i<\cf \kappa}\mu_i$. The following is straightforward:

\begin{proposition}\label{littlegoodness} Suppose $\alpha<\kappa^+$, $\cf \alpha < \cf \kappa$, $A,B \subset \alpha$ are unbounded and $\ot A =\ot B = \cf \alpha$. Then $\vec f_A =^\ast \vec f_B$.\end{proposition}


\begin{proposition}\label{exactgoodness} If $\cf \alpha > \cf \kappa$ and $A \subset \alpha$ witnesses goodness, then $\vec f_A$ is an exact upper bound of $\seq{f_\beta}{\beta \in A}$.\end{proposition}

\begin{proof} It is straightforward that $\vec f_A$ is an upper bound. For exactness, suppose that $g<^\ast \vec  f_A$. Let $j < \lambda$ witness goodness with respect to $A$ as well as $g<^\ast \vec  f_A$, and for all $i$ with $j \le i < \lambda$, let $\beta_i \in A$ be such that $g(i) < f_{\beta_i}(i)$. If $\beta = \sup_{j \le i < \lambda} \beta_i$, then by goodness we have $g <^\ast f_\beta$.\end{proof}

\begin{remark} If $\cf \alpha \le \cf \kappa$, then $\seq{f_\beta}{\beta<\alpha}$ has no exact upper bound: Let $\seq{\beta_\xi}{\xi<\cf \alpha}$ be increasing and cofinal in $\alpha$ and let $\seq{S_\xi}{\xi<\cf \alpha}$ be a partition of $\cf \kappa$ into disjoint unbounded sets. Define $g$ such that $g(i) = f_{\beta_\xi}(i)$ if and only if $i \in S_\xi$. Then $g <^\ast f_\alpha$, but there is no $\beta<\alpha$ such that $g <^\ast f_\beta$.\end{remark}

\begin{proposition}\label{accurategoodness} If $\cf \alpha > \cf \kappa$, $A \subset \alpha$ witnesses goodness of $\alpha$, and $A' \subset A$ is unbounded in $\alpha$, then $\vec f_A =^\ast \vec f_{A'}$.\end{proposition}

\begin{proof} It is immediate that $\vec f_{A'} \le^\ast \vec f_A$. Suppose for contradiction that $\vec  f_{A'} <^\ast \vec  f_A$ as witnessed by $j<\lambda$. Assume that $j$ is also large enough to witnesses goodness with respect to $A$, which implies that it witnesses goodness with respect to $A'$ as well. Then for all $i$ with $j \le i < \lambda$, there is some $\beta_i \in A$ such that $\vec  f_{A'}(i) < f_{\beta_i}(i) < \vec  f_A(i)$. Let $\beta$ be an element of $A'$ greater or equal to $\sup_{j \le i < \lambda}\beta_i < \alpha$. By goodness of $A'$, $i \ge j$ implies that $f_{\beta_i}(i) \le f_\beta(i)$, and so we have $f_\beta(i) \le \vec  f_{A'}(i) < f_\beta(i)$, a contradiction.\end{proof}

\begin{proposition}\label{biggoodness} Suppose $\alpha < \kappa^+$,  $\cf \alpha > \cf \kappa$ and $A,B \subset \alpha$ both witness goodness of $\alpha$. Then $\vec f_A  =^\ast \vec f_B$.\end{proposition}

\begin{proof} Assume that $j$ is large enough to witness goodness with respect to both $A$ and $B$. Use the Sandwich Argument from \autoref{thinnedgoodness} to find $A' \subset A$ and $B' \subset B$ such that $\vec f_{A'} =^\ast \vec f_{B'}$. Our result then follows from \autoref{accurategoodness}.\end{proof}

\begin{proposition}\label{weirdgoodness} Suppose $\alpha < \kappa^+$, $\cf \alpha = \cf \kappa$, and $C,D$ are both clubs in $\alpha$ such that $\ot C = \ot D = \cf \alpha$. Then $f^\Delta_C =^*_\Delta f_D^\Delta$.\end{proposition}

\begin{proof} Suppose otherwise. Enumerate $C = \seq{\beta_i}{i<\cf \kappa}$ and $D = \seq{\gamma_i}{i<\cf \kappa}$. Then without loss of generality, $\{i< \cf \kappa : \vec f^\Delta_C(i) < \vec f^\Delta_D(i)\}$ is stationary in $\cf \kappa$. Let $E$ be the club $\{i < \cf \kappa: \forall j_1,j_2 < i, \exists j^* < i \text{ witnessing } f_{\gamma_{j_1}} <^* f_{\gamma_{j_2}}\}$. Observe that if $i \in \lim E$, then $\seq{f_{\gamma_j}(i)}{j<i}$ is strictly increasing, so for all $\delta < \sup_{j<i}f_{\gamma_j}(i)$, there is some $j'<i$ such that $\delta < f_{\gamma_{j'}}(i)$. Let $S := \lim E \cap \{i< \cf \kappa : \vec f^\Delta_C(i) < \vec f^\Delta_D(i)\}$.

Then for all $i \in S$, there is some $j<i$ such that $\vec f^\Delta_C(i) < f_{\gamma_j}(i)$. By Fodor's Lemma, there is a stationary $T \subset S$ and some $k < \cf \kappa$ such that for all $i \in  T$, $\vec f_C^\Delta(i) < f_{\gamma_k}(i)$. If $\ell$ is large enough that $\gamma_k < \beta_\ell$, then there is some $m$ such that for all $i \ge m$, $f_{\gamma_k}(i) < f_{\beta_\ell}(i)$. If $i > m, \ell$, then $f_{\gamma_k}(i) < f_{\beta_\ell}(i) \le \vec f_C^\Delta(i)$. But $T$ is of course unbounded, so this implies that we can find an $i$ such that $f_{\gamma_k}(i) <  \vec f_C^\Delta(i)  < f_{\gamma_k}(i)$, a contradiction.\end{proof}




\begin{proof}[Proof of \autoref{totally-cont-scale}]We are working with a product $\vec \kappa := \prod_{i<\lambda}\mu_i$. Let $\vec g = \seq{g_\alpha }{ \alpha < \kappa^+}$ be a good scale on this product. Then we define a totally continuous scale $\vec f = \seq{f_\alpha}{\alpha<\kappa^+}$ by induction as follows using the propositions from this section: If $\alpha = \beta+1$, choose $\gamma<\kappa^+$ large enough that $f_\beta <^* g_\gamma$. Then let $f_\alpha$ be such that $g_\gamma <^* f_\alpha$. If $\alpha$ is a limit and $\cf \alpha < \lambda$, choose any $A$, a cofinal subset of $\alpha$ of order-type $\cf \alpha$. Then let $f_\alpha := \vec f_A$. (\autoref{littlegoodness}.) If $\alpha$ is a limit and $\cf \alpha = \lambda$, choose $A$ to be any club subset of $\alpha$ of order-type $\cf \alpha$. Then let $f_\alpha := \vec f^\Delta_A$. (\autoref{weirdgoodness}.) Lastly, suppose $\alpha$ is a limit and $\cf \alpha > \lambda$. Then $\alpha$ is a good point in terms of $\seq{f_\beta}{\beta < \alpha}$ because it is cofinally interleaved with $\seq{g_\beta}{\beta < \alpha}$. Hence we can choose any cofinal $A \subset \alpha$ and let $f_\alpha := \vec f_A$. (\autoref{exactgoodness} and \autoref{biggoodness}.) \end{proof}

\subsection{The Construction for Weak Square}\label{sec-weaksq}

Commencing with the proof of \autoref{terse-weak-sq}, fix a singular $\kappa$ with cofinality $\lambda>\omega$ such that $S^* := \{\delta<\kappa: \square^*_\delta \textup{ holds}\}$ is stationary (and of order-type $\lambda$). It will be sufficient to assume that for all $\tau<\kappa$, $\tau^\lambda<\kappa$, and to assume that $\prod_{\delta \in S^*} \delta^+$ carries a good pseudo-scale.

An easy argument using Fodor's Lemma yields:

\begin{proposition}\label{nice-club} There is a club $E \subset \kappa$ consisting of singular cardinals.\end{proposition}


Using \autoref{nice-club}, let $\seq{\kappa_i}{i<\lambda}$ be a continuous, cofinal, and strictly increasing sequence of singular cardinals in $\kappa$. It follows that $S:= \{i<\lambda: \kappa_i \in \lim(S^*)\}$ is stationary in $\lambda$. Note that $\prod_{i \in S}\kappa_i^+$ also carries a good pseudo-scale, so we can use \autoref{totally-cont-scale} to find a totally continuous pseudo-scale $\vec f =\seq{f_\alpha}{\alpha < \kappa^+}$ on the same product.

Let $\vec{\mathcal C}_i = \seq{\mathcal{C}_\xi^i}{\xi<\kappa_i^+}$ witness $\square^*_{\kappa_i}$ for all $i \in S$. Since $\kappa_i$ is a limit cardinal for all $i$, we can assume that for all such $i$ these $\square^*_{\kappa_i}$-sequences have the property that $\ot C < \kappa_i$ for all $C \in \mathcal C_\xi^i$, $\xi<\kappa_i^+$ (see \cite{Cummings2005}). If $\alpha < \kappa^+$, we define $\mathcal F_\alpha$ as follows:

\begin{itemize}

\item If $\cf(\alpha) \ne \lambda$, we let $\mathcal{F}_\alpha$ be the set of functions $F$ such that $\dom F = S$ and such that $\forall i \in S$, $F(i) \in \mathcal{C}_{f_\alpha(i)}^i$.

\item If $\cf(\alpha) = \lambda$, we  let $\mathcal{F}_\alpha$ be the set of functions $F$ such that $\dom F = S$ and such that \emph{for some} $h =^*_\Delta f_\alpha$, $\forall i \in S$, $F(i) \in \mathcal{C}_{h(i)}^i$.

\end{itemize}

Regardless of whether or not $\cf(\alpha)=\lambda$, we will say that some $h \in \prod_{i \in S}\kappa_i^+$ \emph{witnesses} $F \in \mathcal{F}_\alpha$ if for all $i \in S$, $F(i) \in \mathcal{C}_{h(i)}^i$.

For each $\alpha < \kappa^+$ and $F \in \mathcal F_\alpha$, we define $C_F \subset \alpha$ as follows:

\begin{itemize}

\item If $\beta < \alpha$ and $\cf \beta \ne \lambda$, then $\beta \in C_F$ if and only if there is some $j<\lambda$ such that for all $i \in S \setminus j$, $f_\beta(i) \in \lim F(i)$.

\item If $\beta < \alpha$ and $\cf \beta = \lambda$, then $\beta \in C_F$ if and only if the set of limit ordinals $\gamma \in C_F$ with $\cf(\gamma)<\lambda$ is unbounded in $\beta$.


\end{itemize}

Now we define our $\square_\kappa^*$-sequence at $\alpha$ depending on the cofinality:

\begin{itemize}
\item If $\cf \alpha < \lambda$, then $\mathcal C_\alpha := \{C_F: F \in \mathcal F_\alpha \text{ and }C_F \text{ is unbounded in }\alpha \} \cup \{C \subset \alpha : C \text{ is a club in }\alpha\text{ and }\ot C < \lambda\}$.
\item If $\cf \alpha = \lambda$, choose a club $C \subset \alpha$ such that $\ot C = \lambda$ and let $\mathcal C_\alpha := \{C_F: F \in \mathcal F_\alpha  \text{ and }C_F \text{ is unbounded in }\alpha \} \cup \{C\}$.
\item If $\cf \alpha > \lambda$, let $\mathcal C_\alpha := \{C_F: F \in \mathcal F_\alpha\}$.
\end{itemize}

\begin{lemma}\label{weak-sq-closure} For all $\alpha \in \lim(\kappa^+)$ and $C \in \mathcal{C}_\alpha$, $C$ is closed.\end{lemma}

\begin{proof} It is enough to show that for all $\alpha \in \lim(\kappa^+)$ and $F \in \mathcal F_\alpha$, $C_F$ is closed. The proof of this lemma does not depend on whether or not $\cf(\alpha) = \lambda$; that is, it does not depend on whether $F \in \mathcal{F}_\alpha$ is witnessed specifically by $f_\alpha$ or some $h =^\ast_\Delta f_\alpha$. Let $\seq{\beta_\xi}{\xi<\tau} \s C_F$ be a strictly increasing sequence with supremum $\beta<\alpha$ where $\tau$ is regular. For each $\xi<\tau$, let $j_\xi $ witness that $\beta_\xi \in C_F$, i.e. for all $i \ge j$, $f_{\beta_\xi}(i) \in \lim F(i)$.

\emph{Case 1:} $\tau <\lambda$. If $j'= \sup_{\xi<\tau}j_\xi$, then for all $i \ge j'$, we have $\sup_{\xi<\tau} f_{\beta_\xi}(i) \in \lim F(i)$. By continuity, there is also some $j''$ such that for all $i \ge j''$, $f_\beta(i) = \sup_{\xi<\tau} f_{\beta_\xi}(i)$. Hence, if $j$ is larger than $j'$ and $j''$, then $j$ witnesses that $\beta \in C_F$ by closure of $F(i)$ for $i \in S$.

\emph{Case 2:} $\tau > \lambda$. By the Pigeonhole Principle there is some unbounded $Z \subset \tau$ and some $j'<\lambda$ such that $j_\xi = j'$ for all $\xi \in Z$. By \autoref{thinnedgoodness}, there is some $j''$ and some $Z' \subset Z$ such that $\{\beta_\xi : \xi \in Z'\}$ and $j''$ witness goodness. It then follows by continuity that for all $i \ge j''$, $f_\beta(i) = \sup_{\xi \in Z'}f_{\beta_\xi}(i)$. If $j \ge j',j''$, then $j$ witnesses that $\beta \in C_F$ as in the previous case.

\emph{Case 3:} $\tau = \lambda$. By Case 2, we can assume that $\cf(\beta_\xi) < \lambda$ for all $\xi<\lambda$. Then closure follows by definition. \end{proof}



\begin{lemma}\label{unbounded} For all $\alpha \in \lim(\kappa^+)$, $C \in \mathcal C_\alpha$ is unbounded in $\alpha$.\end{lemma}

\begin{proof} It is sufficient to show that if $\cf \alpha > \lambda$, then $C_F$ is unbounded in $\alpha$ for an arbitrary $F \in \mathcal F_\alpha$. We will use the fact that $F \in \mathcal{F}_\alpha$ can only be witnessed by $f_\alpha$. Consider some $\bar \alpha<\alpha$. We will find an element of $C_F$ larger than $\bar \alpha$. By induction we define a sequence of ordinals $\seq{\alpha_n}{n<\omega}$ in the interval $(\bar \alpha,\alpha)$, a $<^\ast$-increasing sequence of functions $\seq{g_n}{n<\omega}$ in $\prod_{i \in S}\kappa_i^+$, and an undirected list of ordinals $\seq{j_n}{n<\omega}$ in $\lambda$.




Let $\alpha_0 \in (\bar{\alpha},\alpha)$ and $g_0 <^* f_\alpha$ be arbitrary. Suppose that $\alpha_n$ and $g_n$ are defined. Let $g_{n+1}$ be defined so that for all $i<\lambda$, $g_{n+1}(i)$ is an element of $F(i)$ larger than $f_{\alpha_n}(i)$. Using the facts that $g_{n+1}<^\ast f_\alpha$ and that $f_\alpha$ is an exact upper bound of $\seq{f_\beta}{\beta<\alpha}$, find $\alpha_{n+1}$ so that $g_{n+1} <^\ast f_{\alpha_{n+1}}$, and let $j_{n+1}<\lambda$ witness this.

Let $\beta = \sup_{n<\omega}\alpha_n$, which in particular is larger than $\bar \alpha$. We claim that $\beta \in C_F$ as witnessed by $j:=\sup_{n<\omega}j_n<\lambda$. For each $i<\lambda$ such that $i \ge j$, $\seq{g_n(i)}{i<\omega}$ and $\seq{f_{\alpha_n}(i)}{n<\omega}$ interleave each other, so $\sup_{n<\omega}f_{\alpha_n}(i) \in \lim F(i)$ for such $i$. For sufficiently large $i$, $f_{\beta}(i) = \sup_{n<\omega}f_{\alpha_n}(i)$ by continuity, so this completes the proof.\end{proof}

\begin{lemma}\label{coherence} For all $\alpha \in \lim(\kappa^+)$ and $C \in \mathcal C_\alpha$, if $\beta \in \lim C$, then $C \cap \beta \in \mathcal C_\beta$.\end{lemma}

\begin{proof} The lemma is only substantial if $C = C_F$ for some $F \in \mathcal F_\alpha$, and it does not depend on whether $\cf(\alpha)=\lambda$. By assumption $C_F$ is unbounded in $\beta$, so \autoref{weak-sq-closure} implies that $\beta \in C_F$.

\emph{Case 1:} $\cf \beta \ne \lambda$: Let $j<\lambda$ witness $\beta \in C_F$, meaning that if $i \ge j$ then $f_\beta(i) \in \lim F(i)$. By the coherence of $\vec{\mathcal C}^i$ for $i \in S$, it follows that $F(i) \cap f_\beta(i) \in {\mathcal C}_{f_\beta(i)}^i$ for such $i$. Let $F'$ be a function with domain $S$ such that $F'(i) \in \mathcal C_{f_\beta(i)}^i$ for all $i \in S$ and such that $F'(i) = F(i) \cap f_\beta(i)$ for $i \ge j$ in particular. Then $F' \in \mathcal F_\beta$ and $C_{F'}$ is unbounded in $\beta$, so $C_{F'} \in \mathcal C_\beta$. If $\gamma < \beta$, let $j'<\lambda$ witness $f_\gamma <^\ast f_\beta$. Then if $i \ge j,j'$, it follows that $f_\gamma(i) \in F(i)$ if and only if $f_\gamma(i) \in F'(i)$. We conclude that $C_F \cap \beta = C_{F'}$.

\emph{Case 2:} $\cf \beta = \lambda$: Choose a sequence $\seq{\beta_i}{i<\lambda} \subset C_F \cap \beta$; by closure (\autoref{weak-sq-closure}, Case 1) we can assume that $\seq{\beta_i}{i<\lambda}$ is closed and unbounded in $\lambda$, and that $\cf(\beta_i)<\lambda$ for all $i<\lambda$. By \autoref{weirdgoodness}, we also know that $f_\beta =^*_\Delta (\vec f \rest \beta)^\Delta_{\seq{\beta_i}{i<\lambda}}$, i.e$.$ that there is a club $E \subset \lambda$ such that for all $i \in E$, $f_\beta(i) = \sup_{j<i}f_{\beta_j}(i)$. Let $D$ be a club such that $D \subseteq E$ and such that for all $i \in D,j<i$, there is some $j'<i$ witnessing that $\beta_j \in C_F$, and moreoever such that for all $i \in D,j_1,j_2<i$, there is some $j<i$ witnessing that $f_{\beta_{j_1}}<^* f_{\beta_{j_2}}$. It follows that for all $i \in D$ and $j<i$, $f_{\beta_j}(i) \in \lim F(i)$, and therefore that for all $i \in D$, $f_\beta(i) \in \lim F(i)$. Then let $F'$ be defined so that $F'(i) = F(i) \cap f_{\beta}(i)$ for $i \in D \cap S$ and $F'(i) = F(i)$ for $i \in S \setminus D$. Then it follows that $C_F \cap \beta = C_{F'}$: in particular, if $\gamma \in C_{F'}$, then $f_\gamma$ is dominated by $f_\beta$ on a club, so it must be the case that $\gamma<\beta$. Hence we find that $F' \in \mathcal{F}_\beta$ is witnessed by $h$ such that $h(i)=f_\beta(i)$ for $i \in D$ and $h(i) = h'(i)$ for the $h'$ witnessing $F \in \mathcal{F}_\alpha$ (hence $h =^*_\Delta f_\beta$). Therefore we have shown that $C_F \cap \beta \in \mathcal{C}_\beta$.\end{proof}



\begin{lemma}\label{weak-sq-small-ot} For all $\alpha \in \lim(\kappa^+)$ and $C \in \mathcal C_\alpha$, $\ot C < \kappa$.\end{lemma}

\begin{proof} It is sufficent to show that $\ot C_F < \kappa$ for all $F \in \mathcal F_\alpha$ and all $\alpha < \kappa^+$ (independently of whether $\cf(\lambda) = \alpha$). Recall that we assumed that the $\square_{\kappa_i}^*$-sequences $\seq{\mathcal C_\xi^i}{i<\kappa_i^+}$ were defined so that for all for all $i<\lambda, \xi<\kappa_i^+,C \in \mathcal C_\xi^i$, $\ot C <\kappa_i$.

Fix $\alpha<\kappa^+$. For every $i \in S$, there is some $j<i$ such that $\ot F(i) < \kappa_j$. This means that there is a stationary $T \s S$ and some $k$ such that for all $i \in T$, $\ot F(i) < \kappa_k$. If $\beta \in C_F$ and $i \in T$, let $g_\beta(i) = \ot (F(i) \cap f_\beta(i))$ for all $i$ such that $f_\beta(i) \in F(i)$ and $0$ otherwise. The set $\{g_\beta:\beta \in C_F\}$ has size $\kappa_k^{\lambda}<\kappa$ (we assumed this bit of cardinal arithmetic), so it is enough to observe that if $\beta,\beta' \in C_F$ and $\beta<\beta'$, then $g_\beta$ and $g_{\beta'}$ are distinct.\end{proof}


\begin{lemma} For all $\alpha \in \lim(\kappa^+)$, $|\mathcal C_\alpha| \le \kappa$.\end{lemma}

\begin{proof} Our assumption that $\tau^\lambda < \kappa$ for all $\tau < \kappa$ implies that $|\{C \subset \alpha : \ot C < \lambda\}| = \kappa$, so it is enough to show that $|\{C_F:F \in \mathcal{F}_\alpha\}| \le \kappa$ for all $\alpha \in \lim(\kappa^+$).

Fix $\alpha \in \lim(\kappa^+)$. We first argue for the case in which $\cf(\alpha) \ne \lambda$. For all $i \in S$ enumerate $\mathcal C^i_{f_\alpha(i)} = \seq{C^i_\zeta}{\zeta<\kappa_i}$. Given a stationary set $T \subset S$ and $\zeta < \kappa$, let
\[
X_T^k = \{F \in \mathcal{F}_\alpha : \forall i \in T, \exists \zeta < \kappa_k \text{ such that }F(i) = C^i_\zeta \text{ and }\ot(C^i_\zeta) < \kappa_k\}.
\]
We claim that for all $F \in \mathcal F_\alpha$, there are $T \subset S$ and $k < \lambda$ such that $C_F \in X_T^k$. Let $F \in \mathcal F_\alpha$. For each $F$ and $i \in S$, there is some $j<i$ such that we have $F(i) = C^i_\zeta$ for some $\zeta < \kappa_j$ and $\ot(C^i_\zeta)<\kappa_j$ as well. It follows that there is a stationary $T \subset S$ and $k < \lambda$ such that for all $i \in T$, $F(i) = C^i_\zeta$ and $\ot(C^i_\zeta)<\kappa_k$ for some $\zeta < \kappa_k$.

Because $2^\lambda = \lambda^\lambda < \kappa$, there are at most $\kappa$-many $X_T^k$'s. Therefore it remains to show that for all such $T,k$, that $|\{C_F:F \in X_T^k\}| \le \kappa$. Let $G_F$ be the set of functions $g_\beta=f_\beta \rest T$ for all $\beta \in C_F$. If $\beta \ne \beta'$, then $g_\beta \ne g_{\beta'}$, so if $F' \ne F$ then $G_F \ne G_{F'}$. Now, for $i \in T$, let $R^k_T(i) = \bigcup_{\zeta<\kappa_k}C_\zeta^i$. Then for all $F \in F^k_T$, $G_F \subseteq \prod_{i \in T} R^k_T(i)$. Moreover, $\prod_{i \in T} R^k_T(i)$ has cardinality $\kappa_k^\lambda < \kappa$. It follows that $|\{C_F:F \in X_T^k\}| \le \kappa$.

 
Now we comment on the case in which $\cf(\alpha) = \lambda$. By intersecting with an appropriate club, we can see that for all $F \in \mathcal{F}_\alpha$, there is some stationary $S' \subset S$ such that for all $i \in S'$, $F(i) \in C^i_{f_\alpha(i)}$. The argument above can be done for all $F \in \mathcal{F}_\alpha$ such that there is an $h$ witnessing $F \in \mathcal{F}_\alpha$ where $h \rest S' = f_\alpha \rest S'$. Since $2^\lambda < \kappa$, and we only need to consider $2^\lambda$-many possible $S'$, this is sufficient.\end{proof}
 

This finishes the proof of \autoref{terse-weak-sq}.

\subsection{Sketching the Construction for Partial Square}\label{sec-partsq}

Now we will sketch a proof of \autoref{almost-silver-partial-sq}.

This can be proved with the same techniques as the previous theorem, and the setup is basically the same:  We fix a singular strong limit $\kappa$ with cofinality $\lambda>\omega$ such that $\{\delta<\kappa: \square_\delta \textup{ holds}\}$ is stationary (and of order-type $\lambda$). Let $\seq{\kappa_i}{i<\lambda}$ be continuous, cofinal, and strictly increasing in $\kappa$. We find that $S:= \{i<\lambda: \square_{\kappa_i} \textup{ holds}\}$ is stationary in $\lambda$, and we can construct a totally continuous scale $\vec f = \seq{g_\alpha}{\alpha<\kappa^+}$ on $\prod_{i \in S} \kappa_i^+$. Let $\mathcal C_i = \seq{C_\xi^i}{\xi<\kappa_i^+}$ witness $\square_{\kappa_i}$ for all $i \in S$. By \autoref{nice-club}, we can again assume that $\ot C_\xi^i<\kappa_i$ for all $\xi<\kappa_i^+,i<\lambda$. Now we can define the clubs of which our square sequence will consist. For each $\alpha \in \kappa^+ \cap \cof(>\lambda)$, let:
\[
X_\alpha := \seq{\beta<\alpha}{\{i<\lambda:f_\beta(i) \in \lim C^i_{f_\alpha(i)}\} \text{ is co-bounded in }S}.
\]
Then we have an analog of \autoref{weak-sq-closure}:

\begin{lemma}\label{closure} For all $\alpha \in \lim(\kappa^+)$, if $\seq{\beta_\xi}{\xi<\tau} \subset X_\alpha$'s and $\tau \ne \lambda$, then $\sup_{\xi < \tau}\beta_\xi \in X_\alpha$.\end{lemma}

%
%

Then let $C_\alpha$ be the closure of $X_\alpha$ inside $\alpha$. The partial square sequence will be the sequence $\seq{C_\alpha}{ \alpha \in \kappa^+ \cap \cof(>\lambda)}$. Proofs of the various lemmas are analogous. Coherence for the case $\cf(\beta)=\lambda$ is easier since no witness needs to be constructed.

\begin{lemma} For all $\alpha,\beta \in \lim(\lambda^+)$, if $\gamma \in \lim C_\alpha, \lim C_\beta$, then $C_\alpha \cap \gamma = C_\beta \cap \gamma$.\end{lemma}


\begin{proof} Suppose $\gamma$ has cofinality $\tau$.

\emph{Case 1:} $\tau \ne \mu$. Then $X_\alpha$ and $X_\beta$ are both unbounded in $\gamma$, so \autoref{closure} implies that $\gamma \in X_\alpha \cap X_\beta$. We argue as in \autoref{coherence} to show that $X_\alpha \cap \gamma = X_\beta \cap \gamma$, so the result follows.

\emph{Case 2:} $\tau = \mu$. Let $\seq{\gamma_\xi}{\xi<\tau}$ be a strictly increasing and continuous sequence converging to $\gamma$ such that for all $\xi<\tau$, $C_\alpha \cap (\gamma_\xi,\gamma_{\xi+1})$ and $C_\beta \cap (\gamma_\xi,\gamma_{\xi+1})$ are both non-empty. For each limit $\eta < \tau$, $C_\alpha \cap \gamma_\eta = C_\beta \cap \gamma_\eta$ by Case 1. Therefore, $C_\alpha \cap \gamma = C_\beta \cap \gamma$.\end{proof}

\begin{lemma} For all $\alpha \in \lim(\lambda^+)$ such that $\cf \alpha>\cf \lambda$, $C_\alpha$ is unbounded in $\alpha$.\end{lemma}

 It is sufficient to show that if $\cf \alpha > \cf \lambda$, then $X_\alpha$ is unbounded in $\alpha$. We argue as in \autoref{unbounded}, but this version is easier.

\begin{lemma} For all $\alpha \in \lim(\lambda^+)$, $\ot C_\alpha < \lambda$.\end{lemma}

It is sufficent to show that $\ot X_\alpha < \lambda$ for all $\alpha$.

\subsection{Further Directions}

We conclude with the following:

\begin{question} Suppose that $\kappa$ is a singular strong limit of uncountable cofinality $\lambda$ such that $S:=\{\delta < \kappa : \square_\delta^* \text{ holds}\}$. Does $\prod_{\delta \in S} \delta^+$ carry a good pseudo-scale?\end{question}


Note that we have:

\begin{fact} If $\kappa$ is singular, then $\square^*_\kappa$ implies that all pseudo-scales on $\kappa$ are good.\footnote{This is in Cummings' survey \cite{Cummings2005}, but without the distinction involving pseudo-scales.}\end{fact}

By \autoref{terse-weak-sq}, this question is almost equivalent (modulo a generalization and a strong limit assumption) to the question of Golshani mentioned above: a positive answer would mean that these hypotheses imply $\square_\kappa^*$, and a negative answer would mean that $\square_\kappa^*$ consistently fails in conjunction with these hypotheses.


\subsection*{Acknowledgements}

I proved an early version of \autoref{almost-silver-partial-sq} while being supported by Sy-David Friedman's FWF grant in Vienna, and I want to thank him for many helpful conversations. I also thank Dima Sinapova for many helpful conversations and critical readings of early versions of the paper. Also, I thank Chris Lambie-Hanson for catching mistakes in the original arXiv version and for referring me to a note that Assaf Rinot had written that built on this work \cite{Rinot-note2022}. Finally, I thank Sittinon Jirattikansakul for catching additional mistakes in the original version. Thanks also to the anonymous referee for improving the manuscript.

\bibliographystyle{plain}
\bibliography{bibliography}

\begin{thebibliography}{10}

\bibitem{Handbook-Abraham-Magidor}
Uri Abraham and Menachem Magidor.
\newblock Cardinal arithmetic.
\newblock In Matthew Foreman and Akihiro Kanamori, editors, {\em Handbook of
  Set Theory}, pages 1149--1227. Springer, 2010.

\bibitem{Cummings2005}
James Cummings.
\newblock Notes on singular cardinal combinatorics.
\newblock {\em Notre Dame Journal of Formal Logic}, 46(3), 2005.

\bibitem{Cummings-Foreman-Magidor2001}
James Cummings, Matthew Foreman, and Menachem Magidor.
\newblock Squares, scales, and stationary reflection.
\newblock {\em Journal of Mathematical Logic}, 1:35--98, 2001.

\bibitem{Cummings-Foreman-Magidor2003}
James Cummings, Matthew Foreman, and Menachem Magidor.
\newblock The non-compactness of square.
\newblock {\em J. Symbolic Logic}, 68(2):637--643, 2003.

\bibitem{Cummings-Foreman-Magidor2004}
James Cummings, Matthew Foreman, and Menachem Magidor.
\newblock Canonical structure in the universe of set theory {I}.
\newblock {\em Ann. Pure Appl. Logic}, 129(1-3):211--243, 2004.

\bibitem{Golshani-Question}
Mohammad~Golshani (https://mathoverflow.net/users/11115/mohammad golshani).
\newblock Analogues of {S}ilver's theorem for tree property.
\newblock MathOverflow.
\newblock URL:https://mathoverflow.net/q/267270 (version: 2017-04-16).

\bibitem{Jensen1972}
Ronald Jensen.
\newblock The fine structure of the constructible hierarchy.
\newblock {\em Annals of Mathematical Logic}, 4:229--308, 1972.

\bibitem{Krueger2013}
John Krueger.
\newblock Namba forcing and no good scale.
\newblock {\em Journal of Symbolic Logic}, 78(3):785–802, 2013.

\bibitem{LambieHanson-Silver}
Chris Lambie-Hanson.
\newblock A galvin--hajnal theorem for generalized cardinal characteristics.
\newblock {\em European Journal of Mathematics}, 9(1):12, 2023.

\bibitem{Levine-Mildenberger2021}
Maxwell Levine and Heike Mildenberger.
\newblock Distributivity and minimality in perfect tree forcings for singular
  cardinals.
\newblock To appear in \emph{Israel Journal of Mathematics}.

\bibitem{Magidor1982}
Menachem Magidor.
\newblock Reflecting stationary sets.
\newblock {\em The Journal of Symbolic Logic}, 47(4):775--771, 1982.

\bibitem{Rinot-note2022}
Assaf Rinot.
\newblock A note on the ideal {$I[S;\lambda]$}.
\newblock Unpublished note, 2022.

\bibitem{Sargsyan2014}
Grigor Sargsyan.
\newblock Nontame mouse from the failure of square at a singular strong limit
  cardinal.
\newblock {\em Journal of Mathematical Logic}, 14(1):47, 2014.

\bibitem{Schimmerling1995}
Ernest Schimmerling.
\newblock Combinatorial principles in the core model for one {W}oodin cardinal.
\newblock {\em Annals of Pure and Applied Logic}, 74:153--201, 1995.

\end{thebibliography}

\end{document}